\documentclass[11pt]{amsart}

\usepackage{amsthm,amsmath,amssymb}
\usepackage{url}
\newcommand{\C}{{\mathbb{C}}}
\newcommand{\F}{{\mathbb{F}}}
\newcommand{\N}{{\mathbb{N}}}
\newcommand{\Q}{{\mathbb{Q}}}

\newcommand{\cA}{{\mathfrak{A}}}

\newtheorem{thm}{Theorem}
\newtheorem{lem}[thm]{Lemma}
\newtheorem{prop}[thm]{Proposition}

\theoremstyle{definition}
\newtheorem{rem}[thm]{Remark}
\newtheorem{defn}[thm]{Definition}
\newtheorem{exmp}[thm]{Example}

\begin{document}

\title[Jordan--Chevalley decomposition]{On the 
Jordan--Chevalley decomposition of a matrix}

\author{Meinolf Geck}
\address{IDSR - Lehrstuhl f\"ur Algebra, Universit\"at Stuttgart,
Pfaffenwaldring 57, 70569 Stuttgart, Germany}
\email{meinolf.geck@mathematik.uni-stuttgart.de}

\subjclass[2000]{Primary 15A21, Secondary 11C99}
\keywords{Matrices, polynomials, Jordan--Chevalley decomposition, 
computer algebra}

\date{May 12, 2022}

\begin{abstract} The purpose of this note is to advertise an elegant
algorithmic proof for the Jordan--Chevalley decomposition of a matrix,
following and (slightly) revising the discussion of Couty, Esterle und
Zarouf (2011). The basic idea of that method goes back to Chevalley (1951).
\end{abstract}

\maketitle

\section{Introduction}

Let $K$ be a field and $M_n(K)$ be the vector space of $n\times n$-matrices
with entries in $K$. Let $A\in M_n(K)$ be such that the characteristic 
polynomial of~$A$ splits into linear factors over $K$. Then we can write 
uniquely $A=D+N$ where $D,N\in M_n(K)$ are such that $D$ is diagonalisable, 
$N$ is nilpotent and $D\cdot N=N\cdot D$. This is called the 
\textit{Jordan--Chevalley decomposition} of~$A$. Usually, this is deduced 
from the Jordan normal form of $A$ (which is a stronger result). However, 
there is a direct, short algorithmic proof for the existence of $D,N$ 
which is inspired by the Newton iteration for finding roots of a function. 
This goes back to Chevalley and is discussed in Couty et al.~\cite{cez} 
but, as the authors write, it does not seem to be widely known (especially
not in teaching normal forms of matrices in Linear Algebra).

The purpose of this note is, firstly, to advertise that proof which yields 
an elegant procedure for producing $D$ and~$N$, without even knowing the 
eigenvalues of $A$ (which may be a considerable advantage in some 
situations). Secondly, we somewhat reduce the required prerequisites 
about polynomials. The algorithm in \cite{cez} uses the ``square-free 
part'' of a non-constant polynomial~$f\in K[X]$, the usual 
definition of which relies on the prime factorisation of~$f$. But in basic 
courses on Linear Algebra, the prime factorisation of polynomials is often 
not yet available. And even if it is, then its use in concrete examples 
is not obvious because, in general, the prime factorisation of a polynomial
is difficult to compute. Our aim is to present the algorithm in a way that 
avoids that completely; all we shall need is the formal derivative and 
the Euclidean Algorithm for polynomials in $K[X]$.

Thirdly, a {\sf GAP} \cite{gap4} program for computing the Jordan--Chevalley
decomposition, based on the algorithm in this note and extending the 
functionality already developed in \cite{frob}, is now available in the
file {\tt frobenius.g} at 

\centerline{\url{https://pnp.mathematik.uni-stuttgart.de/idsr/idsr1/geckmf/}}

\noindent This program appears to work quite well even for large matrices, 
especially when combined with the efficient algorithm for computing the 
Frobenius normal form of matrices in \cite{frob}; see Example~\ref{ex1} at 
the very end of this note.

\section{Square-free polynomials}
Let $K[X]$ be the polynomial ring over $K$ in the indeterminate $X$. We 
define the ``formal derivative'' as the linear map $D\colon K[X]\rightarrow 
K[X]$ such that $D(1)=0$ and $D(X^i)=iX^{i-1}$ for all $i\geq 1$. The 
usual product rule holds in this case as well (as one easily checks): 
\[D(fg)=D(f)g+fD(g)\qquad \mbox{for all $f,g\in K[X]$}.\]
Furthermore, let $f,g\in K[X]$ be such that $f\neq 0$ or $g\neq 0$. Then 
the Euclidean Algorithm yields the existence of a unique monic greatest 
common divisior $d=\mbox{gcd}(f,g) \in K[X]$, as well as polynomials $r,s
\in K[X]$ such that $d=rf+sg$. We say that $f,g$ are \textit{coprime} if 
$\mbox{gcd}(f,g)=1$.~---~This is all we shall require in the following 
discussion. 

\begin{defn} \label{def1} Let $f\in K[X]$ be non-zero. We say that $f$ is
\textit{square-free} if $f$ and $D(f)$ are coprime, that is, $\mbox{gcd}
(f,D(f))=1$. In particular, a non-zero constant polynomial is 
square-free. (See also Remark~\ref{rem22} below.)
\end{defn}

\begin{lem} \label{lem1} Let $f,g\in K[X]$ be non-zero.
\begin{itemize}
\item[(a)] If $f$ is square-free and $g\mid f$, then $g$ is also 
square-free.
\item[(b)] If $f,g$ are both square-free and coprime, then $fg$ is also 
square-free.
\item[(c)] If $f,g$ are square-free and $m\in K[X]$ is a least common
multiple of $f$ and $g$, then $m$ is also square-free.
\end{itemize}
\end{lem}

\begin{proof} (a) Write $f=gh$ with $h\in K[X]$. Let $0\neq d\in K[X]$ 
be monic such that $d\mid g$ and $d\mid D(g)$. Since $g\mid f$, we have 
$d\mid f$. Since $D(f)=D(g)h+gD(h)$, we also have $d\mid D(f)$. Hence, 
$d\mid D(f)$ and $d\mid f$; so $d=1$.

(b) Let $0\neq d\in K[X]$ be monic such that $d\mid fg$ and $d\mid D(fg)$. 
Let $d_1=\mbox{gcd}(d,f)$ and $d_2:=\mbox{gcd}(d,g)$. Now $d_1\mid d$, 
$d\mid D(f)g+fD(g)$ and $d_1\mid f$, so $d_1\mid D(f)g$. Since $d,g$ are 
coprime it is clear that $d_1$, $g$ are coprime. So we must have $d_1\mid 
D(f)$ (see \cite[Remark~4.1(c)]{frob}) and, hence, $d_1=1$; that is, $d$ 
and~$f$ are coprime. Similarly, one sees that $d$ and $g$ are coprime. So 
$d$ and $fg$ are coprime by \cite[Remark~4.1(a)]{frob}.

(c) By \cite[Lemma~4.3]{frob}, there exist $f_1,g_1\in K[X]$ such that 
$f_1\mid f$, $g_1\mid g$, $m=f_1g_1$ and $f_1,g_1$ are coprime. So $m$ is 
square-free by (a) and (b).
\end{proof}

The field $K$ is called \textit{perfect} if either $\mbox{char}(K)=0$
or $p:=\mbox{char}(K)>0$ and the map $K\rightarrow K$, $x\mapsto x^p$, 
is surjective. In addition to fields of characteristic~$0$, all finite 
fields are known to be perfect; every algebraically closed field is,
of course, perfect.

\begin{prop} \label{prop1} Assume that $K$ is perfect. Let $f\in K[X]$ be 
non-constant. Then there exists a square-free $g\in K[X]$ such that
$g\mid f$ and $f\mid g^m$ for some $m\in\N$.  
\end{prop}

\begin{proof} We proceed by induction on $\mbox{deg}(f)\geq 1$.
If $\mbox{deg}(f)=1$, then the assertion holds with $g:=f$ and $m:=1$.
Now let $\mbox{deg}(f)\geq 1$. 

First assume that $D(f)\neq 0$; let $f_1:=\mbox{gcd}(f,D(f))\in K[X]$. 
If $f_1$ is constant, then $f$ is square-free and the assertion holds again 
with $g:=f$ and $m:=1$. Otherwise write $f=f_1f_2$ with $f_2\in K[X]$. 
Since $D(f)\neq 0$, we certainly have $\mbox{deg}(f_1)\leq \mbox{deg}(D(f)) 
<\mbox{deg}(f)$. Since $\mbox{deg}(f_1) \geq 1$, we also have $\mbox{deg}
(f_2)<\mbox{deg}(f)$. Hence, by induction, there exist square-free $g_i\in 
K[X]$ such that $g_i\mid f_i$ and $f_i\mid g_i^{m_i}$ for $i=1,2$, where 
$m_i\in \N$. Let $g\in K[X]$ be a least common multiple of $g_1,g_2$. Then 
Lemma~\ref{lem1} shows that $g$ is square-free. Since $g_i\mid f$ for
$i=1,2$, we also have $g\mid f$. Since $f_i\mid g_i^{m_i}$ for $i=1,2$,
we have $f\mid g_1^{m_1}g_2^{m_2}$ and so $f\mid g^m$ for $m:=m_1+m_2$. 

Now assume that $D(f)=0$. Write $f=\sum_{i=0}^n a_iX^i$ where $n=
\mbox{deg}(f)\geq 1$ and $a_i \in K$ for all~$i$. Then $p:=
\mbox{char}(K)>0$ and $a_i=0$ for all $i$ with $p\nmid i$. So $n=pn'$ 
for some $n' \in \N$ and $f=\sum_{j=0}^{n'} a_{pj} X^{jp}$. Since $K$ is 
perfect, we can write $a_{pj}=b_j^p$ with $b_j\in K$. Then $f=h^p$ where
$h:=\sum_{j=0}^{n'} b_jX^j\in K[X]$. By induction, there exists a 
square-free $g\in K[X]$ such that $g\mid h$ and $h\mid g^{m'}$ for some
$m'\in K[X]$. Then $g\mid f$ and $f\mid g^m$ with $m:=m'p$.
\end{proof}

\begin{rem} \label{rem1} Given a non-constant $f\in K[X]$, the above proof 
provides an efficient algorithm for computing the required square-free $g 
\in K[X]$. It is a somewhat more precise version of the procedure described 
in \cite[p.~130]{ln}. (If $\mbox{char}(K)=0$, then it is known that one 
can just take $g=f/\mbox{gcd}(f,D(f))$.) Also note that a least common
multiple of two non-zero $f,g\in K[X]$ is obtained as $fg/\mbox{gcd}(f,g)$.
\end{rem}

\begin{rem} \label{rem22} Assume that $f\in K[X]$ splits into linear
factors in $K$. That is, we can write $f=c(X-\lambda_1)^{n_1}\cdots(X-
\lambda_r)^{n_r}$ where $0\neq c\in K$ and $\lambda_1,\ldots,\lambda_r
\in K$ are the distinct roots of $f$ in $K$; furthermore, $n_i\geq 1$ for
all~$i$. Then one does not need to assume that $K$ is perfect in order to
produce $g\in K[X]$ and $m\in \N$ as in Proposition~\ref{prop1}. Indeed, 
just set $g:=(X-\lambda_1)\cdots (X-\lambda_r)\in K[X]$ and $m:=\max\{n_1,
\ldots,n_r\}$. One easily computes $D(g)$ and checks that $D(g)(\lambda_i)
\neq 0$ for all~$i$; hence, $\mbox{gcd}(g,D(g))=1$.
\end{rem}

\section{The Jordan--Chevalley decomposition}
Let us now fix a matrix $A\in M_n(K)$. We say that $A$ is \textit{nilpotent} 
if $A^m=0$ for some $m\in \N$. We say that $A$ is \textit{semisimple} if 
$A$ is diagonalisable (possibly only over some larger field $L\supseteq 
K$). Let $f\in K[X]$ be non-constant such that $f(A)=0$. 
Standard candidates for $f$ are the characteristic polynomial $\chi_A
\in K[X]$ of $A$ (by the Theorem of Cayley--Hamilton), or the minimal 
polynomial $\mu_A\in K[X]$ of~$A$. For example, the matrix 
\[ A=\left[\begin{array}{cc} 0 & -1 \\ 1 & 0 \end{array}\right]\in M_2(\Q)
\qquad \mbox{with} \qquad \chi_A=X^2+1\in \Q[X],\]
is semisimple, since it is diagonalisable over $\C\supseteq \Q$.

It will be convenient to set $\cA:=\{r(A)\mid r \in K[X]\}\subseteq 
M_n(K)$; note that all elements of $\cA$ commute with each other. 
If $B\in \cA$ is fixed, we denote $B\cdot \cA=\{B\cdot r(A)\mid 
r\in K[X]\}$. 

\begin{lem} \label{lem31} Let $B,C\in \cA$ and $h\in K[X]$. Then we have:
\[h(B+C)-h(B)\in C\cdot \cA\;\;\mbox{and}\;\;h(B+C)-h(B)-D(h)(B)\cdot 
C \in C^2\cdot \cA.\]
\end{lem}

\begin{proof} It is enough to prove this for $h=X^m$ where $m\geq 0$.
In this case, the assertion immediately follows from the binomial formula
for $(B+C)^m$.
\end{proof}

Given $f\in K[X]$ as above with $f(A)=0$, assume that there is a
square-free $g\in K[X]$ such that $g\mid f$ and $f\mid g^m$ for some
$m\in \N$. (If $K$ is perfect, this holds by Lemma~\ref{lem1}; if $f$ 
splits into linear factors, see Remark~\ref{rem22}.) Since $g,D(g)$ are 
coprime, there exist $\tilde{g}, \tilde{g}'\in K[X]$ such that $1=
\tilde{g}D(g)+ \tilde{g}'g$; note that $\tilde{g}\neq 0$. We define a 
sequence of matrices $(A_k)_{k\geq 0}$ in $M_n(K)$ by 
\begin{center}
\fbox{$\;A_0:=A\qquad \mbox{and}\qquad A_{k+1}:=A_k-g(A_k)\cdot 
\tilde{g}(A_k) \quad \mbox{for $k\geq 0$}.\;$}
\end{center}

\begin{thm}[Cf.\ Couty et al. \protect{\cite[Th\'eor\`eme 1]{cez}}]
Let $k_0\geq 0$ be such that $2^{k_0}\geq m$. Then $g(A_{k_0})=0$, 
$A_{k_0}$ is semisimple and $A-A_{k_0}$ is nilpotent; furthermore,
$A_{k_0}\in\cA$. Thus, $A=D+N$ with $D:=A_{k_0}$ and $N:=A-A_{k_0}$ is 
the Jordan--Chevalley decomposition of~$A$.
\end{thm}

\begin{proof} We somewhat re-organise the proof in \cite{cez} and 
check that everything still works in our slightly different setting.
First note that, by a simple induction on $k$, we have $A_k\in \cA$
for all $k\geq 0$. The crucial step is to show:
\begin{equation*}
A-A_k\in g(A)\cdot \cA \quad\mbox{and}\quad g(A_k)\in g(A)^{2^k}\cdot\cA
\quad \mbox{for all $k\geq 0$}.\tag{$*$}
\end{equation*}
We use induction on $k$. For $k=0$ we have $A=A_0$ and ($*$) is clear.
Now let $k\geq 0$. We write $A_{k+1}=A_k+B$ where $B:=-g(A_k)\cdot 
\tilde{g}(A_k)\in g(A_k)\cdot \cA$. By induction, $g(A_k)\in g(A)^{2^k}
\cdot \cA\subseteq g(A)\cdot \cA$ and so $B\in g(A)\cdot \cA$. Also 
by induction, $A-A_k\in g(A)\cdot\cA$ and, hence, $A-A_{k+1}=(A-A_k)-B 
\in g(A)\cdot\cA$. Now consider $g(A_k)$. By the second formula in 
Lemma~\ref{lem31}, we have 
\[g(A_{k+1})-g(A_k)-D(g)(A_k)\cdot B\in B^2\cdot \cA.\]
Using now the identity $1=\tilde{g}D(g)+\tilde{g}'g$, we obtain 
\begin{align*}
-D(g)&(A_k)\cdot B=D(g)(A_k)\cdot \tilde{g}(A_k)\cdot g(A_k)\\ &= 
(1-\tilde{g}'g)(A_k)\cdot g(A_k)=g(A_k)-\tilde{g}'(A_k)\cdot g(A_k)^2.
\end{align*}
This yields $g(A_{k+1})-\tilde{g}'(A_k)\cdot g(A_k)^2\in B^2\cA$. Since 
$g(A_k)\in g(A)^{2k}\cdot \cA$ and $B\in g(A_k)\cdot \cA$, we deduce 
that $g(A_{k+1})\in g(A)^{2^{k+1}}\cdot \cA$, as required.

Having established ($*$), we can now conclude as follows. Since 
$2^{k_0}\geq m$, we have $f\mid g^m\mid g^{2^{k_0}}$ and so $g^m(A)=
g^{2^{k_0}}(A)=f(A)=0$. Hence, ($*$) implies that $g(A_{k_0})=0$ and 
$(A-A_{k_0})^m=0$. Now it is known that we can always find a field $L
\supseteq K$ such that $g\in K[X]$ splits into linear factors
over~$L$. Since $g$ is square-free, that is, $\mbox{gcd}(g,D(g))=1$, it 
follows that~$g$ does not have any repeated roots in $L$.  (If $\lambda
\in L$ was such a repeated root, then $g=(X-\lambda)^2h$ for some 
$h\in K[X]$ and so $D(g)=2(X-\lambda)h+(X-\lambda)^2D(h)$ would also 
be divisible by $X-\lambda$, contradiction.) Hence, $A$ will be 
diagonalisable over~$L$. (This is a standard result in Linear Algebra 
and can be established directly, without knowledge of the Jordan normal 
form; see, for example, \cite[Theorem~7.16]{fis}.) Finally, we have 
already noted that $A_{k_0}\in \cA$.

Uniqueness also follows by a standard argument. Suppose that we also 
have $A=D'+N'$ where $D'$ is semisimple, $N'$ is nilpotent and
$D'\cdot N'=N'\cdot D'$. Then consider the identity $D-D'=N'-N$. 
Since $D',N'$ commute with each other, they commute with $A$. Since
$A_{k_0}\in \cA$, we also have $D,N\in \cA$. Hence, $D,D'$ commute with 
each other, and $N,N'$ commute with each other. But then $N'-N$ will 
also be nilpotent and so $(D-D')^d=0$ for some $d\in\N$. Since $D,D'$ 
are diagonalisable (over some $L\supseteq K$) and commute, they can in 
fact be simultaneously diagonalised (see, for example, Exercise~39 in 
\cite[\S 5.4]{fis}). Hence $D-D'$ can also be diagonalised. But then 
$(D-D')^d=0$ implies $D-D'=0$; so $D=D'$ and also $N=N'$.
\end{proof}

%In the above proof, we have used the general fact that every polynomial
%in $K[X]$ splits into into linear factors over some larger field
%$L\supseteq K$. In a basic course on Linear Algebra, this result will not 
%yet be available. But then one could just state the theorem for matrices 
%whose characteristic polynomial splits into linear factors over $K$.

\begin{exmp} \label{ex1} The {\sf GAP} package mentioned in the 
Introduction now also provides the function {\tt JordanChevalleyDecMat}, 
which takes as input a matrix $A\in M_n(K)$ and a non-constant polynomial 
$f\in K[X]$ such that $f(A)=0$; the output are $D,N$ as above. 

For the example matrix $U\in M_{15}(\C)$ in \cite[\S 3]{cez} (for which
the eigenvalues can only be computed approximately), we obtain $D,N$ in a 
few milli-seconds. 

Now let $K=\F_2$ (the field with~$2$ elements) and 
$A=M_2\in M_{4370}(\F_2)$ be the (challenging) test matrix already considered 
at the end of \cite{frob}. The {\sf GAP} function {\tt MinimalPolynomial} 
produces the minimal polynomial $\mu_A\in \F_2[X]$ in a few seconds; we 
have $\mbox{deg}(\mu_A)=2097$. The square-free part $g\in \F_2[X]$ of 
$\mu_A$ has degree~$2087$. The Euclidean Algorithm quickly yields
$\tilde{g},\tilde{g}' \in \F_2[X]$ such that $1=\tilde{g} D(g)+ 
\tilde{g}'g$; here, $\mbox{deg}(\tilde{g})=2085$ and $\mbox{deg}
(\tilde{g}')=2084$. The function {\tt JordanChevalleyDecMat} then 
takes about 5 minutes to compute~$D,N$; the matrix $N$ has minimal 
polynomial $X^5$. 

Two further remarks:

1) One has to expect that the computation for $A=M_2$ as above takes a 
considerable time, because it requires the evaluation of $g(A_k)$ and 
$\tilde{g}(A_k)$ for $k=1,2,3$, which involves thousands of matrix 
multiplications in $M_{4370}(\F_2)$. In order to perform this in 
5 minutes, we actually first compute the Frobenius normal form of 
$A$ and then apply the function to each diagonal block in that normal form;
these diagonal blocks are quite sparse and, hence, computations will be 
much faster. We provide a function {\tt JordanChevalleyDecMatF} which 
does this automatically. (Applying {\tt JordanChevalleyDecMat} 
directly to~$A$ takes around 45 minutes.)

2) We do obtain the Jordan--Chevalley decomposition in this example, but
it is practically impossible to compute the Jordan normal form. This is 
because the square-free polynomial $g\in\F_2[X]$ (which is the minimal 
polynomial of the semisimple part $D$ of $A$) has irreducible factors of 
degree $1$, $2$, $4$, $6$, $88$, $197$, $854$, $934$. So, although exact 
computations are possible in finite fields, one would need to go to an 
enormous field $L\supseteq \F_2$ in order to write down the eigenvalues 
of~$A$.
\end{exmp}

%%%%%%%%%%%%%%%%%%%%%%%%%%%%%%%%%%%%%%%%%%%%%%%%%%%%%%%%%%%%%%%%%%%%%%%

%\bigskip
%\noindent {\itshape\scriptsize IDSR -- Lehrstuhl f\"ur Algebra, Universit\"at 
%Stuttgart, Pfaffenwaldring 57, 70569 Stuttgart, Germany\\ 
%meinolf.geck@mathematik.uni-stuttgart.de}
\end{document}